\documentclass[11pt]{article}
\usepackage[english]{babel}
\usepackage[cp1251]{inputenc}
\usepackage{amscd,amssymb,amsmath,amsthm}
\newtheorem{thm}{Theorem}
\newtheorem{lemma}{Lemma}

\newtheorem{cor}{Corollary}

\begin{document}

\sloppy

\begin{center}
\textbf {CONTRACTING QUADRATIC OPERATORS OF BISEXUAL POPULATION}\\[2mm]
\begin{center}
Nasir N. Ganikhodjaev\\[2mm]

Department of Computational and Theoretical Sciences, Faculty of Science,IIUM,\\[2mm]
25200 Kuantan, Malaysia.\\[2mm]
Email: nasirgani@hotmail.com\\[2mm]

Uygun U. Jamilov \\[2mm]
Department of Computational and Theoretical Sciences, Faculty of Science,IIUM,\\[2mm]
25200 Kuantan, Malaysia.\\[2mm]
Institute of Mathematics at the National University Uzbekistan,
Tashkent, Uzbekistan,\\[2mm]
 e-mail: {\tt jamilovu@yandex.ru, ujamilov@rambler.ru}
\end{center}
\end{center}

 \vskip 0.5 truecm

{\bf Abstract.} In this paper we find a sufficient condition under which the
operator of bisexual population is contraction and show that this condition is not necessary.

\vskip 0.5 truecm


\vskip 0.5 truecm


\vskip 0.8 truecm

\section{Introduction}

    The action of genes is manifested statistically in sufficiently
large communities of matching individuals (belonging to the same
species). These communities are called populations \cite{Lyu}. The
population exists not only in space but also in time, i.e. it has
its own life cycle. The basis for this phenomenon is reproduction
by mating. Mating in a population can be free or subject to
certain restrictions.

    The whole population  in space and time comprises discrete
generations $F_0,F_1,...$. The generation $F_{n+1}$ is the set of
individuals whose parents belong to the $F_n$ generation. A state
of a population is a distribution of probabilities of the
different types of organisms in every generation. Type partition
is called differentiation. The simplest example is sex
differentiation. In bisexual population any kind of
differentiation must agree with the sex differentiation, i.e. all
the organisms of one type must belong to the same sex. Thus, it is
possible  to speak of male and female types.

    The evolution (or dynamics) of a population comprises a
determined change of state in the next generations as a result of
reproductions and selection. This evolution of a population can be
studied by a dynamical system (iterations) of a quadratic
stochastic operator.

    The history of the quadratic stochastic operators  can be traced
back to the  work of S. Bernshtein \cite{Br}. For more than 80
years this theory has been developed and many papers were
published (see \cite{Br}-\cite{Zh}). Several problems of physical
and biological systems lead to necessity of study the asymptotic
behavior of the trajectories
of quadratic stochastic operators.\\

    Let  $E=\{1,2,...,m\}$. By the $(m-1)-$ simplex we mean the
set
\begin{equation} \label{simp1}
S^{m-1}=\{\textbf{x}=(x_1,...,x_m)\in R^m: x_i\geq 0, \sum^m_{i=1}x_i=1 \}.
\end{equation}
    Each element $\textbf{x}\in S^{m-1}$ is a probability measure on $E$ and so
it may be looked upon as the state of a biological (physical and
so on) system of $m$ elements.\\

    A quadratic stochastic operator $V:S^{m-1}\rightarrow S^{m-1}$
has the form
\begin{equation} \label{kso2}
V: x_k'=\sum^m_{i,j=1}p_{ij,k}x_ix_j, \ \ (k=1,...,m),
\end{equation}
where  $p_{ij,k}-$  coefficient of heredity and
\begin{equation} \label{koef3}
 p_{ij,k}=p_{ji,k} \geq 0, \ \ \sum^m_{k=1}p_{ij,k}=1, \ \ (i,j,k=1,...,m).
\end{equation}

For a given  $x^{(0)}\in S^{m-1}$, the trajectory  $\{x^{(n)}\}, \
\ n=0,1,2,...$ of $x^{(0)}$ under the action of QSO (\ref{kso2})
is defined by $x^{(n+1)}=V(x^{(n)}),$ where $n=0,1,2,...$\\

One of the main problems in mathematical biology is to study the
asymptotic behavior of the trajectories. There are many papers
devoted to study of the evolution of the free population, i.e. to
study of dynamical system generated by quadratic stochastic
operator (\ref{kso2}), see e.g. \cite{GR1}-\cite{ZhM}. In \cite{GRRU} a survey
of theory quadratic stochastic operators is given.\\

In this paper we find a condition under which the evolutionary
operators of bisexual population is contraction.\\

\section{Definitions}

In this section following \cite{Lyu}, we describe the evolution
operator of a bisexual population. Assuming that the population is
bisexual we suppose that the set of females can be partitioned
into finitely many different types indexed by $\{1,2,...,n\}$ and,
similarly, that the male types are indexed by $\{1,2,...,\nu\}$.
The number $n+\nu$ is called the dimension of the population. The
population is described by its state vector $(\textbf{x},\textbf{y})$ in
$S^{n-1}\times S^{\nu-1}$, the product of two unit simplexes in
$\mathbb{R}^n$ and $\mathbb{R}^\nu$ respectively. Vectors $\textbf{x}$ and
$\textbf{y}$ are the probability distributions of the females and males
over the possible types:
\begin{equation} \label{simpbp4}
x_i\geq0,\  \ \sum_{i=1}^{n}x_i=1;\  \  y_j\geq0,\  \
\sum_{j=1}^{\nu}y_j=1.
\end{equation}

    Denote $S=S^{n-1}\times S^{\nu-1}$. We call the partition into
types hereditary if for each possible state $\textbf{z}=(\textbf{x},\textbf{y})\in S$
describing the current generation, the state
$z^\prime=(x^\prime,y^\prime)\in S$ is uniquely defined describing
the next generation. This means that the association $z\rightarrow
z^\prime$ defined a map $V:S\rightarrow S$ called the evolution
operator.
    For any point $z^{(0)}\in S$ the sequence $z^{(t)}=V(z^{(t-1)}),
t=1,2,...$ is called the trajectory of $z^{(0)}$.
    Let $p_{ij,k}^{(f)}$ and $p_{ij,l}^{(m)}$ be inheritance
coefficients defined as the probability that a female offspring is
type $k$ and, respectively, that a male offspring is type $l$,
when the parental pair is $ij(i,k=1,2,...,n;$ and
$j,l=1,2,...,\nu)$. We have
\begin{equation} \label{koefbp5}
p_{ij,k}^{(f)}\geq 0 , \sum \limits_{k=1}^n p_{ij,k}^{(f)} =1, \
\  p_{ij,l}^{(m)}\geq 0 , \sum \limits_{l=1}^\nu p_{ij,l}^{(m)}=1.
\end{equation}

    Let $z^\prime=(x^\prime,y^\prime)$ be the state of the offspring
population at the birth stage. This is obtained from inheritance
coefficients as

\begin{equation} \label{ksobp5}
W:\left\{\begin{array}{ll}
x^\prime_k=\sum \limits_{i,j=1}^{n,\nu} p_{ij,k}^{(f)}x_iy_j,\  \ (1\leq k \leq n)\\[4mm]
y^\prime_l= \sum \limits_{i,j=1}^{n,\nu} p_{ij,l}^{(m)}x_iy_j,\ \
(1\leq l \leq \nu).
\end{array}\right.
\end{equation}

    We see from (\ref{ksobp5}) that for a bisexual population the
evolution operator is a quadratic mapping of $S$ into itself. But
for free population the operator is quadratic mapping of the
simplex into itself given by (\ref{kso2}).

    In \cite{RL} an algebra of the bisexual population is defined as the
following:\\
Consider $\{e_1,...,e_{n+\nu}\}$ the canonical basis on
$\mathbb{R}^{n+\nu}$ and divide the basis as $e^{(f)}_i=e_i,
i=1,...,n$ and $e^{(m)}_i=e_{n+i}, i=1,...,\nu$. Introduce on
$\mathbb{R}^{n+\nu}$ a multiplication defined by
\begin{equation} \label{multbp8}
\begin{array}{llll}
e^{(f)}_ie^{(m)}_j=e^{(m)}_je^{(f)}_i=\frac{1}{2}\bigg(
\sum\limits_{k=1}^np^{(f)}_{ij,k}e^{(f)}_{k}+\sum\limits_{l=1}^\nu
p^{(m)}_{ij,l}e^{(m)}_{l}\bigg);\\[2mm]
e^{(f)}_ie^{(f)}_k=0, \  \ i,k=1,...,n;\\[2mm]
e^{(m)}_je^{(m)}_l=0, \  \ j,l=1,...,\nu;\\[2mm]
\end{array}
\end{equation}

Thus the coefficients of bisexual inheritance is the structure
constants of an algebra, i.e. a bilinear mapping of
$\mathbb{R}^{n+\nu}\times\mathbb{R}^{n+\nu}$ to
$\mathbb{R}^{n+\nu}$.
    The general formula for the multiplication is the extension of
(\ref{multbp8}) by bilinearity, i.e. for $z,t\in
\mathbb{R}^{n+\nu}$,
$$
z=(x,y)=\sum\limits_{i=1}^n x_ie^{(f)}_i+\sum\limits_{j=1}^\nu
y_je^{(m)}_j ,  \  \  t=(u,v)= \sum\limits_{i=1}^n
u_ie^{(f)}_i+\sum\limits_{j=1}^\nu v_je^{(m)}_j
$$
using (\ref{multbp8}), we obtain
\begin{equation} \label{multbp9}
\begin{array}{llll}
zt=\frac{1}{2}\sum\limits_{k=1}^n\bigg(
\sum\limits_{i=1}^n\sum\limits_{j=1}^\nu
p^{(f)}_{ij,k}(x_iv_j+u_iy_j)\bigg)e^{(f)}_{k}+\\[2mm]
\ \ + \frac{1}{2}\sum\limits_{l=1}^\nu \bigg(
\sum\limits_{i=1}^n\sum\limits_{j=1}^\nu
p^{(m)}_{ij,l}(x_iv_j+u_iy_j)\bigg)e^{(m)}_{l}.
\end{array}
\end{equation}

From (\ref{multbp9}) and using (\ref{ksobp5}), in the particular
case that $z=t$, i.e. $x=u$ and $y=v$, we obtain
$$
zz=z^2=\sum\limits_{k=1}^n\bigg(
\sum\limits_{i=1}^n\sum\limits_{j=1}^\nu
p^{(f)}_{ij,k}x_iy_j\bigg)e^{(f)}_{k}+
$$
$$
\ \ +\sum\limits_{l=1}^\nu \bigg(
\sum\limits_{i=1}^n\sum\limits_{j=1}^\nu
p^{(m)}_{ij,l}x_iy_j\bigg)e^{(m)}_{l}=W(z).
$$
for any $z\in S$.
    This algebraic interpretation is very useful. For example, a
bisexual population state $z=(x,y)$ is an equilibrium (fixed
point) precisely when $z$ is an idempotent element of the
set $S$, i.e. $z=z^2$.

    The algebra $\mathcal{B}=\mathcal{B}_W$ generated by the evolution operator $W$ (see
(\ref{ksobp5})) is called the {\it evolution algebra of the bisexual
population}.

    In \cite{RL} it was shown that if $z$ is a fixed point then $z\in
R^{n+\nu}_0 \bigcup R^{n+\nu}_1$, where
\begin{equation} \label{hpbp6}
R^{n+\nu}_\eta=\{z=(x,y):\sum\limits_{i=1}^nx_i=\sum\limits_{j=1}^{\nu}y_j=\eta\},
\ \ \eta=0,1.
\end{equation}
    For simplex $S=S^{n-1}\times S^{\nu-1}$  by  tangent space we get
\begin{equation} \label{hpbp7}
R^{n+\nu}_0=\{z=(x,y):\sum\limits_{i=1}^nx_i=\sum\limits_{j=1}^{\nu}y_j=0\}.
\end{equation}

\section{Contracting operators}

    In operator theory, a bounded operator $W: X \rightarrow Y$ between normed vector spaces
$X$ and $Y$ is said to be a {\it contraction} if its operator norm $\|W\| \leq 1$.

    An extremal example of a quadratic contraction is the constant operator.
In this case the coefficients $ p_{ij,k}^{(f)},p_{ij,l}^{(m)}$ do not depend
on $i$ and $j$. This suggests that for a sufficiently small scattering of
coefficient for every fixed $k,l$ the quadratic operator will be a contraction.
This remark can be expressed as a precise theorem.

    The {\it Lipschitz constant} of an operator  $W:\mathbb{R}^{n+\nu}\rightarrow \mathbb{R}^{n+\nu}$  is
$$
L(W)=\sup\limits_{z\neq t} \frac{\|Wz-Wt\|}{\|z-t\|},
$$
where $\|\cdot\|$ is some norm in $\mathbb{R}^{n+\nu}$. If this norm can be
chosen so that $L(W)<1$ then $W$ will be a strict contraction in this norm
with the consequences: unique fixed point, convergence of all trajectories to
this point, exponential rate of convergence. Unless otherwise specified, we
will use the $l_1-$ norm in the basis $e^{(f)}_i=e_i,
i=1,...,n$ and $e^{(m)}_i=e_{n+i}, i=1,...,\nu$ defined as $\|z\|=\sum\limits_{i=1}^n x_i+\sum\limits_{j=1}^\nu y_j$ for
$z=(x,y)=\sum\limits_{i=1}^n x_ie^{(f)}_i+\sum\limits_{j=1}^\nu
y_je^{(m)}_j.$

\begin{lemma}\label{lemma2}\cite{Lyu}. Let $\Delta$ be a convex $n-$
dimensional compact in $\mathbb{R}^n$, $F:\Delta\rightarrow\Delta$ be a
smooth map. Then (for any norm) $L(F)\equiv \max\limits_{z\in\Delta}
\|d_zF\|$.
\end{lemma}

\begin{lemma}\label{lemma3}\cite{Lyu}. Let a matrix  $A=(a_{ij})_{i,j=1}^n$ satisfies
\begin{equation} \label{matrix12}
\sum\limits_{i=1}^na_{i1}=\sum\limits_{i=1}^na_{i2}=...=\sum\limits_{i=1}^na_{in}.
\end{equation} Then
\begin{equation} \label{matrix14}
\|A|\mathbb{R}^n_0\|=\frac{1}{2}\max\limits_{j_1\neq j_2}\sum\limits_{i=1}^n|a_{ij_1}-a_{ij_2}|,
\end{equation}
where $A|\mathbb{R}^n_0$ is restriction operator $A$ on $\mathbb{R}^n_0$.
\end{lemma}

 For each $z\in \mathcal{B}$ we have linear operator $M_z:\mathcal{B}
\rightarrow \mathcal{B}$ defined by $M_z(t)=zt$.

\begin{thm}\label{thm1} The following inequality holds for the Lipschitz's constant

$$L(W) \leq \max\limits_{i_1,i_2,j}\bigg(\sum\limits_{k=1}^n|p_{i_1j,k}^{(f)}-p_{i_2j,k}^{(f)}|+
\sum\limits_{l=1}^\nu|p_{i_1j,l}^{(m)}-p_{i_2j,l}^{(m)}|\bigg)+$$
\begin{equation} \label{detbp15}
\max\limits_{j_1,j_2,i}\bigg(\sum\limits_{k=1}^n|p_{ij_1,k}^{(f)}-p_{ij_2,k}^{(f)}|+
\sum\limits_{l=1}^\nu|p_{ij_1,l}^{(m)}-p_{ij_2,l}^{(m)}|\bigg).
\end{equation}
\end{thm}

\begin{proof} For the operator $W$ in  $S$ the derivative is
\begin{equation} \label{jacovbp11}
d_zW=\frac{1}{2}\left( \begin{array}{lllllllllllllll}
\sum\limits_{j=1}^\nu p^{(f)}_{1j,1}y_j & ... &
\sum\limits_{j=1}^\nu p^{(f)}_{nj,1}y_j & \sum\limits_{i=1}^n
p^{(f)}_{i1,1}x_i & ... &
\sum\limits_{i=1}^n p^{(f)}_{i\nu,1}x_i \\[2mm]
\  \ \vdots & \ddots & \  \ \vdots & \  \ \vdots & \ddots &\  \  \vdots\\[2mm]
\sum\limits_{j=1}^\nu p^{(f)}_{1j,n}y_j & ... &
\sum\limits_{j=1}^\nu p^{(f)}_{nj,n}y_j & \sum\limits_{i=1}^n
p^{(f)}_{i1,n}x_i & ... &
\sum\limits_{i=1}^n p^{(f)}_{i\nu,n}x_i\\[2mm]
\sum\limits_{j=1}^\nu p^{(m)}_{1j,1}y_j & ... &
\sum\limits_{j=1}^\nu p^{(m)}_{nj,1}y_j & \sum\limits_{i=1}^n
p^{(m)}_{i1,1}x_i & ... &
\sum\limits_{i=1}^n p^{(m)}_{i\nu,1}x_i \\[2mm]
\  \ \vdots & \ddots & \  \ \vdots & \  \ \vdots & \ddots &\  \  \vdots\\[2mm]
\sum\limits_{j=1}^\nu p^{(m)}_{1j,\nu}y_j & ... &
\sum\limits_{j=1}^\nu p^{(m)}_{nj,\nu}y_j & \sum\limits_{i=1}^n
p^{(m)}_{i1,\nu}x_i & ... &
\sum\limits_{i=1}^n p^{(m)}_{i\nu,\nu}x_i\\[2mm]
\end{array}\right)
\end{equation}
$$d_zW=2M_z=2\sum\limits_{k=1}^nx_kM^{(f)}_k+2\sum\limits_{l=1}^\nu y_lM^{(m)}_l,$$
where $M_k^{(f)}=M_{e_k^{(f)}}$ and $M_l^{(m)}=M_{e_l^{(m)}}$ is the multiplication maps with
matrixes $(p^{(f)}_{ij,k})_{i,k=1}^n$ and respectively
$(p^{(m)}_{ij,l})_{j,l=1}^\nu$.\\
By Lemma \ref{lemma2} we have  $L(W)=2\max\limits_{z\in S}
\|M_z\|\leq 2\max\limits_{k}\|M^{(f)}_k\|+2\max\limits_{l}\|M_l^{(m)}\|$.\\

By Lemma \ref{lemma3},
$$\|M_k^{(f)}\|=\frac{1}{2}\max\limits_{i_1,i_2,j}\bigg(\sum\limits_{k=1}^n|p_{i_1j,k}^{(f)}-p_{i_2j,k}^{(f)}|+
\sum\limits_{l=1}^\nu|p_{i_1j,l}^{(m)}-p_{i_2j,l}^{(m)}|\bigg),
$$
$$
\|M_l^{(m)}\|=\frac{1}{2}\max\limits_{j_1,j_2}\bigg(\sum\limits_{k=1}^n|p_{ij_1,k}^{(f)}-p_{ij_2,k}^{(f)}|+
\sum\limits_{l=1}^\nu|p_{ij_1,l}^{(m)}-p_{ij_2,l}^{(m)}|\bigg).
$$
\end{proof}

\begin{cor}\label{cor1} An evolutionary operator (\ref{ksobp5}) is a strict contraction if
$$
\max\limits_{i_1,i_2,j}\bigg(\sum\limits_{k=1}^n|p_{i_1j,k}^{(f)}-p_{i_2j,k}^{(f)}|+
\sum\limits_{l=1}^\nu|p_{i_1j,l}^{(m)}-p_{i_2j,l}^{(m)}|\bigg)+$$
\begin{equation} \label{norma15}
\max\limits_{j_1,j_2,i}\bigg(\sum\limits_{k=1}^n|p_{ij_1,k}^{(f)}-p_{ij_2,k}^{(f)}|+
\sum\limits_{l=1}^\nu|p_{ij_1,l}^{(m)}-p_{ij_2,l}^{(m)}|\bigg)<1
\end{equation}
\end{cor}

    For evolutionary operators with positive coefficients there is a multiplicative estimate
of the distance from  the evolutionary operator to the  constant one.
Let
$$
\mu^{f}\equiv \mu^{f}(W)=\max\limits_{i_1,i_2,j,k}\frac{p^{(f)}_{i_1j,k}}{p^{(f)}_{i_2j,k}}, \  \
\mu^{m}\equiv \mu^{m}(W)=\max\limits_{i,j_1,j_2,l}\frac{p^{(m)}_{ij_1,l}}{p^{(m)}_{ij_2,l}},
$$
and let $\zeta(W)$ equal to  LHS of (15).

\begin{lemma}\label{lemma4}
\begin{equation} \label{matrix16}
\zeta(W)\leq 4\frac{\mu^{f}-1}{\mu^{f}+1}+4\frac{\mu^{m}-1}{\mu^{m}+1}.
\end{equation}
\end{lemma}

\begin{proof}
If $\alpha,\beta>0$ and $\mu=\max(\frac{\alpha}{\beta},\frac{\beta}{\alpha})$ then obviously
$$
|\alpha-\beta|=\frac{\mu-1}{\mu+1}(\alpha+\beta).
$$
Hence
$$
|p_{i_1j,k}^{(f)}-p_{i_2j,k}^{(f)}|\leq\frac{\mu^{f}-1}{\mu^{f}+1}(p_{i_1j,k}^{(f)}+p_{i_2j,k}^{(f)}),\   \
|p_{i_1j,l}^{(m)}-p_{i_2j,l}^{(m)}|\leq\frac{\mu^{m}-1}{\mu^{m}+1}(p_{i_1j,l}^{(m)}+p_{i_2j,l}^{(m)}),
$$
and respectively
$$
|p_{ij_1,k}^{(f)}-p_{ij_2,k}^{(f)}|\leq\frac{\mu^{f}-1}{\mu^{f}+1}(p_{ij_1,k}^{(f)}+p_{ij_2,k}^{(f)}), \  \
|p_{ij_1,l}^{(m)}-p_{ij_2,l}^{(m)}|\leq\frac{\mu^{m}-1}{\mu^{m}+1}(p_{ij_1,l}^{(m)}+p_{ij_2,l}^{(m)}).
$$

It remains to sum these inequalities over $k$ and respectively over $l$, keeping in mind that
$$
\sum\limits_{k=1}^n p_{i_1j,k}^{(f)}=\sum\limits_{k=1}^n p_{i_2j,k}^{(f)}=
\sum\limits_{l=1}^\nu p_{ij_1,l}^{(m)}=\sum\limits_{l=1}^\nu p_{ij_2,l}^{(m)}=1.
$$
\end{proof}

\begin{cor}\label{cor2}
$$L(W)\leq 4\frac{\mu^{f}-1}{\mu^{f}+1}+4\frac{\mu^{m}-1}{\mu^{m}+1}.$$
\end{cor}

\begin{cor}\label{cor3}
If $ \  \  7\mu^f\mu^m-(\mu^f+\mu^m)<9$ then the evolutionary operator (\ref{ksobp5})
is a strict contraction.
\end{cor}

\begin{cor}\label{cor4}
Let  $ \mu= \max(\mu^f,\mu^m)$. Then
$$L(W)\leq 8 \frac{\mu-1}{\mu+1}$$
 and if $\mu<\frac{9}{7}$ then the evolutionary operator (\ref{ksobp5})
is a strict contraction.
\end{cor}

    Let us give several examples and check the condition of Corollary \ref{cor1}.

{\bf Example 1.} Consider the operator
\begin{equation} \label{ksobp17}
W:\left\{\begin{array}{llll}
x^\prime_1=\frac{3}{7}x_1y_1+\frac{1}{2}x_1y_2+\frac{1}{2}x_2y_1+\frac{4}{7}x_2y_2,\\[2mm]
x^\prime_2=\frac{4}{7}x_1y_1+\frac{1}{2}x_1y_2+\frac{1}{2}x_2y_1+\frac{3}{7}x_2y_2,\\[2mm]
y^\prime_1=\frac{4}{7}x_1y_1+\frac{1}{2}x_1y_2+\frac{1}{2}x_2y_1+\frac{3}{7}x_2y_2,\\[2mm]
y^\prime_2=\frac{3}{7}x_1y_1+\frac{1}{2}x_1y_2+\frac{1}{2}x_2y_1+\frac{4}{7}x_2y_2.\\[2mm]
\end{array}\right.
\end{equation}
The coefficients of the operator (\ref{ksobp17}) as the following
$$
\begin{array}{llll}
p_{11,1}^{(f)}=\frac{3}{7} & p_{12,1}^{(f)}=\frac{1}{2} & p_{21,1}^{(f)}=\frac{1}{2} & p_{22,1}^{(f)}=\frac{4}{7}\\[2mm]
p_{11,2}^{(f)}=\frac{4}{7} & p_{12,2}^{(f)}=\frac{1}{2} & p_{21,2}^{(f)}=\frac{1}{2} & p_{22,2}^{(f)}=\frac{3}{7}\\[2mm]
p_{11,1}^{(m)}=\frac{4}{7} & p_{12,1}^{(m)}=\frac{1}{2} & p_{21,1}^{(m)}=\frac{1}{2} & p_{22,1}^{(m)}=\frac{3}{7}\\[2mm]
p_{11,2}^{(m)}=\frac{3}{7} & p_{12,2}^{(m)}=\frac{1}{2} & p_{21,2}^{(m)}=\frac{1}{2} & p_{22,2}^{(m)}=\frac{4}{7}\\[2mm]
\end{array}
$$

    It is easy to check that condition (\ref{norma15}) satisfied for (\ref{ksobp17}). Indeed,

$$
\max\limits_{i_1,i_2,j}\bigg(\sum\limits_{k=1}^n|p_{i_1j,k}^{(f)}-p_{i_2j,k}^{(f)}|+
\sum\limits_{l=1}^\nu|p_{i_1j,l}^{(m)}-p_{i_2j,l}^{(m)}|\bigg)+$$
$$
\max\limits_{j_1,j_2,i}\bigg(\sum\limits_{k=1}^n|p_{ij_1,k}^{(f)}-p_{ij_2,k}^{(f)}|+
\sum\limits_{l=1}^\nu|p_{ij_1,l}^{(m)}-p_{ij_2,l}^{(m)}|\bigg)=\frac{4}{7}.
$$

    Consequently, this operator is a strict contraction and it has unique fixed point
$(\frac{1}{2},\frac{1}{2},\frac{1}{2},\frac{1}{2})$. Moreover any trajectory of (\ref{ksobp17})
converges to the fixed point.

    The following example shows that the condition of Corollary \ref{cor1} is not satisfied and evolutionary
operator has periodic trajectory.

{ \bf Example 2.} Consider the operator
\begin{equation} \label{ksobp18}
W:\left\{\begin{array}{llll}
x^\prime_1=x_1y_1\\[2mm]
x^\prime_2=x_1y_2+x_2\\[2mm]
y^\prime_1=x_2y_2\\[2mm]
y^\prime_2=x_1+x_2y_1\\[2mm]
\end{array}\right.
\end{equation}

 It easy to check that operator (\ref{ksobp18}) does not satisfy the condition of
Corollary \ref{cor1}.

    We rewrite the operator (\ref{ksobp18}) in the form
\begin{equation} \label{ksobp181}
W:\left\{\begin{array}{ll}
x^\prime_1=x_1y_1\\[2mm]
y^\prime_1=(1-x_1)(1-y_1)\\[2mm]
\end{array}\right.
\end{equation}

  Denote $x_n=x_1^{(n)}, \  \ y_n=y_1^{(n)}$ then from (\ref{ksobp181}) we have
\begin{equation} \label{ksobp191}
\left\{\begin{array}{ll}
x_{n+1}=x_ny_n\\[2mm]
y_{n+1}=(1-x_n)(1-y_n)\\[2mm]
\end{array}\right.
\end{equation}

  Since $ 0\leq x_ny_n \leq x_n$ from the first equation of (\ref{ksobp191}) it follows
that $\lim\limits_{n\rightarrow\infty}x_n=x^*=0$. Indeed, for
$(x^0,y^0)\in int ( S^1\times S^1)$ we get from (\ref{ksobp191})
$$
\frac{x_{n+2}}{x_{n+1}}=(1-x_n)\bigg(1-\frac{x_{n+1}}{x_{n}}\bigg),
$$
$$
x_{n+2}x_n=x_{n+1}(1-x_n)(x_n-x_{n+1}),
$$
$$
\lim\limits_{n\rightarrow\infty}x_{n+2}x_n=\lim\limits_{n\rightarrow\infty}x_{n+1}(1-x_n)(x_n-x_{n+1}),
$$
$$
(x^{*})^2=0, \  \  x^{*}=0.
$$
Now consider the operator
\begin{equation} \label{ksobp192}
W^2:\left\{\begin{array}{ll}
x'=xy-x^2y-xy^2+x^2y^2\\[2mm]
y'=x+y-xy-x^2y-xy^2+x^2y^2\\[2mm]
\end{array}\right.
\end{equation}

Clearly, the operator $W^2$ has fixed points  $(0,y)$, $0\leq y\leq 1$. The point
$(0,y)$ is a saddle point.

It is easy to check that the set $\{(x,y)\in S^1\times S^1: x_1=0\}$ is an invariant subset for (\ref{ksobp18}). Any point of the invariant subset
is periodic point with period two for operator (\ref{ksobp18}). So trajectory of the operator with an initial point form invariant subset does not converge.
Thus operator (\ref{ksobp18}) has a trajectory which does non-converge to the fixed point $(0,1,{1\over 2}, {1\over 2})$.\\

    The following example shows that condition of Corollary \ref{cor1} is sufficient but is
not necessary.

{ \bf Example 3.}  Consider the operator with coefficients of inheritance
$$
\begin{array}{llll}
p_{11,1}^{(f)}=0 & p_{12,1}^{(f)}=0 & p_{21,1}^{(f)}=\frac{1}{2} & p_{22,1}^{(f)}=\frac{1}{2}\\[2mm]
p_{11,2}^{(f)}=1 & p_{12,2}^{(f)}=1 & p_{21,2}^{(f)}=\frac{1}{2} & p_{22,2}^{(f)}=\frac{1}{2}\\[2mm]
p_{11,1}^{(m)}=0 & p_{12,1}^{(m)}=\frac{1}{2} & p_{21,1}^{(m)}=0 & p_{22,1}^{(m)}=\frac{1}{2}\\[2mm]
p_{11,2}^{(m)}=1 & p_{12,2}^{(m)}=\frac{1}{2} & p_{21,2}^{(m)}=1 & p_{22,2}^{(m)}=\frac{1}{2}\\[2mm]
\end{array}
$$
i.e.  the evolution operator has the form
\begin{equation} \label{ksobp19}
W:\left\{\begin{array}{llll}
x^\prime_1=\frac{1}{2}x_2\\[2mm]
x^\prime_2=x_1+\frac{1}{2}x_2\\[2mm]
y^\prime_1=\frac{1}{2}y_2\\[2mm]
y^\prime_2=y_1+\frac{1}{2}y_2\\[2mm]
\end{array}\right.
\end{equation}

 It easy to check that operator (\ref{ksobp19}) does not satisfy the condition of
Corollary \ref{cor1}.
$$
\max\limits_{i_1,i_2,j}\bigg(\sum\limits_{k=1}^n|p_{i_1j,k}^{(f)}-p_{i_2j,k}^{(f)}|+
\sum\limits_{l=1}^\nu|p_{i_1j,l}^{(m)}-p_{i_2j,l}^{(m)}|\bigg)+$$
$$
\max\limits_{j_1,j_2,i}\bigg(\sum\limits_{k=1}^n|p_{ij_1,k}^{(f)}-p_{ij_2,k}^{(f)}|+
\sum\limits_{l=1}^\nu|p_{ij_1,l}^{(m)}-p_{ij_2,l}^{(m)}|\bigg)=2>1.
$$
But any trajectory of (\ref{ksobp19}) converges to $(\frac{1}{3}, \frac{2}{3},\frac{1}{3}, \frac{2}{3})$.

Indeed, from (\ref{ksobp19}) we have

$$
x^{(n+1)}_1=\frac{1}{2}(1-x^{(n)}_1)
$$

We consider following one dimensional dynamical system.

$$f(x)=\frac{1}{2}(1-x)$$

It has unique fixed point $x=\frac{1}{3}$  and decreasing on $[0,1]$.
Easy to check that  $f'(x)=-\frac{1}{2}$ and  $|f'(\frac{1}{3})|=\frac{1}{2}<1$ therefore the
fixed point $x=\frac{1}{3}$ is attracting.

We claim that any trajectory of $f(x)$ converges to the fixed point $x=\frac{1}{3}$. Indeed, we have
$$
 f^n(x)=\sum\limits_{k=1}^n \frac{(-1)^k}{2^k}+(-1)^n\cdot\frac{x}{2^n}
$$
and
$$
\lim \limits_{n\rightarrow\infty} f^{2n} (x) = \lim \limits_{n\rightarrow\infty}
\bigg(\frac{1}{3}\cdot\frac{2^{2n}-1}{2^{2n}}+\frac{x}{2^{2n}}\bigg)=\frac{1}{3},
$$
$$
\lim \limits_{n\rightarrow\infty} f^{2n+1} (x) = \lim \limits_{n\rightarrow\infty}
\bigg(\frac{1}{3}\cdot\frac{2^{2n}-1}{2^{2n}}+\frac{1}{2^{2n+1}}-\frac{x}{2^{2n+1}}\bigg)=\frac{1}{3}.
$$

So for any initial point  trajectory of (\ref{ksobp19}) converges to $(\frac{1}{3}, \frac{2}{3},\frac{1}{3}, \frac{2}{3})$.

 \vskip 0.3 truecm

 {\bf Acknowledgment.} The final part of this work was done at the International Islamic  University of Malaysia (IIUM)
and the second author would like to thank the IIUM for providing financial support and all facilities. The second author also thanks
Dr. Mansoor Saburov for useful discussions. \\

 \vskip 0.3 truecm

\begin{center}

\end{center}
\end{document}